\documentclass[12pt, reqno]{amsproc}%
\usepackage[margin=1in]{geometry}
\usepackage{amsmath,amssymb}
\usepackage{amsthm,mathtools,nicefrac}
\usepackage{graphicx,caption,color,multirow}

\usepackage[vcentermath,enableskew]{youngtab}
\usepackage{ytableau}

\renewcommand{\phi}{\varphi}

\usepackage{hyperref}
\hypersetup
{
    colorlinks,	%
    citecolor=red,%
    filecolor=black,%
    linkcolor=blue,%
    urlcolor=blue
}

\providecommand{\U}[1]{\protect\rule{.1in}{.1in}}

\newcounter{theorem}
\newtheorem{thm}[theorem]{Theorem}

\newtheorem{cor}[theorem]{Corollary}
\newtheorem{prop}[theorem]{Proposition}
\newtheorem{lem}[theorem]{Lemma}
\newtheorem{thmn}{Theorem}
\theoremstyle{definition}
\newtheorem{defn}[theorem]{Definition}
\newtheorem{expl}[theorem]{Example}
\newtheorem{rem}[theorem]{Remark}
\newcommand{\al}{\alpha}

\newcommand{\la}{\lambda}
\newcommand{\ka}{\kappa}

\newcommand{\lm}{^\lambda_\mu}

\newcommand{\Q}{\mathbf{Q}}
\renewcommand{\P}{\mathbf{R}}
\newcommand{\M}{\mathbf{M}}
\newcommand{\N}{\mathbf{N}}
\newcommand{\LS}{\mathbf{X}}

\newcommand{\Ja}{P}

\newcommand{\set}[1]{\left\{#1\right\}}

\title[A combinatorial formula for certain binomial coefficients]{A combinatorial formula for certain binomial coefficients for Jack polynomials}
\author{Yusra Naqvi}
\author{Siddhartha Sahi}
\thanks{The research of the second author was partially supported by a Simons Foundation grant (509766). Part of this work was carried out during the Workshop on Hecke Algebras and Lie Theory at the University of Ottawa, co-organized by the second author and attended by the first. The authors thank the National Science Foundation (DMS-162350), the Fields Institute, and the University of Ottawa for funding this workshop.}

\begin{document}

\begin{abstract}
We present a decomposition of the generalized binomial coefficients associated with Jack polynomials into two factors: a {\em stem}, which is described explicitly in terms of hooks of the indexing partitions, and a {\em leaf}, which inherits various recurrence properties from the binomial coefficients and depends exclusively on the skew diagram. We then derive a direct combinatorial formula for the leaf in the special case where the two indexing partitions differ by at most two rows. This formula also exhibits an unexpected symmetry with respect to the lengths of the two rows. 
\end{abstract}

\maketitle

  
\section*{Introduction}

This work is motivated by the long-standing quest for direct, combinatorial
formulas for binomial coefficients associated with {Jack symmetric
polynomials. We recall that Jack polynomials are }multivariate symmetric
polynomials $P_{\lambda }=P_{\lambda }^{(\alpha )}(x_{1},\ldots ,x_{M})$,
which are indexed by partitions $\lambda $ of length $\leq M$, and whose
coefficients belong to the field $\mathbb{F}=\mathbb{Q}\left( \alpha \right) 
$ of rational functions in the parameter $\alpha $. In fact the $P_{\lambda }
$ form a homogeneous basis for the algebra of symmetric polynomials $\mathbb{%
F}\left[ x_{1},\ldots ,x_{M}\right] ^{S_{M}}$, and hold a special place in
the algebraic hierarchy of multivariate symmetric polynomials. For $\alpha
=0,\nicefrac{1}{2},1,2,\infty $ they specialize to the symmetric monomials, quaternionic
zonal polynomials, Schur functions, real zonal polynomials, and elementary
symmetric polynomials, respectively. In turn they, along with the
Hall-Littlewood polynomials, were one of the two sources of inspiration for
Macdonald's definition of his two parameter family of symmetric polynomials 
\cite{macdonald:95}.

For partitions $\lambda $ and $\mu ${, the Jack binomial coefficient} $%
b_{\mu }^{\lambda }=b_{\mu }^{\lambda }\left( \alpha \right) \in \mathbb{F}$
is defined by  
\begin{equation*}
\frac{P_{\lambda }(1+x_{1},\ldots ,1+x_{M})}{P_{\lambda
}(1,\ldots ,1)}=\sum_{\mu }b_{\mu }^{\lambda }\frac{P_{\mu
}(x_{1},\ldots ,x_{M})}{P_{\mu }(1,\ldots ,1)}.
\end{equation*}
In full generality these coefficients were first considered by Lassalle \cite%
{lassalle:90}, although the special cases $\alpha =1$ and $\alpha =2$
occurred in earlier work of Lascoux \cite{lascoux:82} and Bingham \cite%
{bingham:74}. They were also studied extensively by Okounkov and Olshanski 
\cite{okounkov:olshanski:97} who showed that the $b_{\mu }^{\lambda }$ are
special values of the interpolation polynomials. These latter polynomials
were first defined by one of us in \cite{sahi:94}, and studied in \cite%
{knop:97, knop:sahi:96, sahi:11}. We note that there are
anlaogous definition for binomial coefficients for Macdonald symmetric
polynomials \cite{okounkov:97} and for nonsymmetric Jack and Macdonald
polynomials \cite{sahi:98}, and we hope to treat these in future work.

The starting point of the present paper is a recursive formula for $b_{\mu
}^{\lambda }$ recently discovered by one of us \cite{sahi:11}. By way of
background, we recall that to a partition $\lambda =\left( \lambda
_{1},\ldots ,\lambda _{M}\right) $ one attaches a Young diagram consisting
of a left justified array of boxes, with $\lambda _{i}$ boxes in row $i$.
For a box $b$ in $\lambda $ we write arm$_{\lambda }(b)$ and leg$_{\lambda
}(b)$ for the number of boxes to the right of $b$ and below $b$,
respectively. Here and in the subsequent discussion, we will find it
convenient to set%
$$
r=1/\alpha,
$$
and we define the upper and lower $r$-hooks to be 
$$
c_{\lambda }(b)=\text{arm}(b)+r\cdot \text{leg}(b)+r,\quad c_{\lambda }^{\prime
}(b)=\text{arm}(b)+r\cdot \text{leg}(b)+1.
$$

We write $\lambda \supset \mu $ if the Young diagram of $\lambda $ contains
that of $\mu $, i.e. if $\lambda _{i}\geq \mu _{i}$ for all $i$, and in this
case we denote by $\lambda /\mu $ the skew diagram consisting of the boxes
of $\lambda $ not in $\mu $ (the shaded boxes in the picture below). 
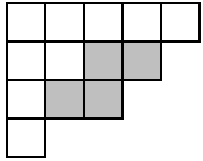
\[
\ytableausetup{boxsize=1.2em}
\ytableaushort{~~~~~,~~~~,~~~,~}*[*(lightgray)]{0,2+2,1+2}
\]
Now a key property of binomial coefficients is that $b_{\mu }^{\lambda }=0$ 
\emph{unless} $\lambda \supset \mu $. If $\lambda /\mu $ consists of a
single box we say that $\lambda ,\mu $ are adjacent, and in this case there
is an explicit combinatorial formula due to Kaneko \cite{kaneko:93} that
expresses $b_{\mu }^{\lambda }$ in terms of upper and lower hooks of $%
\lambda $ and $\mu $. One of the main results of \cite{sahi:11} asserts that
the infinite matrix $\left( b_{\mu }^{\lambda }\right) $ of binomial
cofficients is the exponential of the matrix of adjacent binomial
coefficients, which leads to the recursive formula for the coefficients
alluded to above and recalled in Section \ref{sec:recur} below.

However, finding closed-form formulas for $b_{\mu }^{\lambda }$ remains a
challenging and active area of ongoing research. For example, as noted in 
\cite{sahi:11}, after multiplication by a suitable (and precise) factor the
coefficients seem to be polynomials in $r$ with positive integer
coefficients, but there is currently no proof of this conjecture.

We now briefly describe the main results of this paper. Given partitions
$\lambda\supset\mu$ we label the boxes of $\lambda$ as follows: we label the
boxes of $\lambda/\mu$ by $S$ and label the remaining boxes of $\lambda$ to
indicate whether they share a row $(R)$, column $(C)$, both $(J)$, or neither
$(N)$, with the skew boxes:
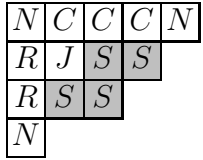
\[
\ytableausetup{boxsize=1.2em}
\ytableaushort{NCCCN,RJSS,RSS,N}*[*(lightgray)]{0,2+2,1+2}
\]
For convenience we write $C,R,$ etc. for the set of boxes with label $C,R,$
etc. We define the \textbf{stem} $K_{\mu }^{\lambda }$ and the \textbf{
leaf }$L_{\mu }^{\lambda }$ as follows 
\[
K_{\mu }^{\lambda }=\left( \prod_{b\in C}\frac{c_{\lambda }(b)}{c_{\mu }(b)}%
\right) \left( \prod_{b\in R}\frac{c_{\lambda }^{\prime }(b)}{c_{\mu
}^{\prime }(b)}\right) \left( \prod_{b\in J}\frac{1}{c_{\mu }(b)c_{\mu
}^{\prime }(b)}\right) ,\quad L_{\mu }^{\lambda }=\frac{b_{\mu }^{\lambda }}{%
K_{\mu }^{\lambda }}.
\]

\begin{thm}
The leaf $L_{\mu}^{\lambda}$ depends only on the skew diagram $\lambda/\mu$.
\label{prop:leafskew}
\end{thm}

Thus, although the Jack binomial coefficient $b_{\mu }^{\lambda }$ can (and
does) vary with $\lambda $ and $\mu $ for fixed $\lambda /\mu$, our result
shows that this variation is explicitly described combinatorially by the
stem.

The leaf still seems to be a fairly complicated combinatorial object. Our
second main result is an explicit formula for leaf in the case where $%
\lambda /\mu $ consists of (at most) two rows. Our analysis breaks up
naturally into two cases --- either the two rows of $\lambda /\mu $ share no
columns whatsoever, or they have a certain number of overlapping columns. We
will refer to these shortly as the ``gap''
and ``overlap'' cases. Somewhat remarkably it turns that it is possible to give a uniform formula that covers both cases. To this end we attach a $4-$tuple $(u,d,m,y)$ to a skew shape $\lambda /\mu$ with at most two rows, as follows: $u$ and $d$ are the number of boxes in the upper and lower rows respectively, and $m$ is the number of overlapping columns. In the gap case (when $m$ is 0) we define $y$ to be
$$
y=y_{\mu }^{\lambda }=\text{arm}_{\mu }(x^{\ast })+r\cdot \text{leg}_{\mu
}(x^{\ast }),
$$ and in the overlap case we set $y=0.$ 

Now, given \emph{any} non-negative integers $u$ and $d\geq m$, let $d^{\prime }=d-m$, and then define:
\begin{equation}
L\left( u,d;m,y\right) =\sum_{\ell=0}^{d^{\prime}}\binom{d^{\prime}}{\ell}%
\prod_{i=0}^{\ell-1}(m+i+1-r)(i+r)\prod_{j=\ell+1}^{d^{\prime}}(y+d^{\prime
}+r-j)(y+u+r+j).  \label{eq:Lud}
\end{equation}

\begin{thm}
\label{thm:main} Let $\mu \subset \lambda $ be partitions such that $\lambda
/\mu $ consists of at most two rows. Then we have 
\[
L_{\mu }^{\lambda }=L\left( u,d;m,y\right),
\]
where $(u,d,m,y)$ is the $4-$tuple associated to $\la/\mu$ as described above.
\end{thm}


We present two examples in Table \ref{tab:ex} below.

\begin{center}
\begin{table}[h!] 
\caption{Two examples of stems and leaves} \label{tab:ex}
\begin{tabular}[5pt]{|r||c|c|}
\hline
& $\la=(7,3,3,1),\mu=(4,3,1,1)$ & $\la=(6,5,3,1),\mu=(6,2,1,1)$ \\
\hline
\hline
&&\\
{} &$\ytableaushort{RJ{x^*}RSSS,NCC,RSS,N}*[*(lightgray)]{4+3,0,1+2}$ & $\ytableaushort{NCCCCN,R{x^*}SSS,RSS,N}*[*(lightgray)]{0,2+3,1+2}$\\
&&\\
\hline
$u$ & 3 & 3 \\
$d$ & 2 & 2 \\
$m$ & 0 & 1 \\
$y$ & $1+r$ & 0 \\
$K^\la_\mu$ & $\displaystyle \frac{4(7+3r)}{(4+3r)(1+r)^3(2+r)}$ & $\displaystyle \frac{3(4+3r)(1+2r)(5+2r)}{r(4+2r)(2+r)(1+r)}$\\
$L^\la_\mu$ & $60+238r+275r^2+90r^3+9r^4$ & $6r$ \\
\hline
\end{tabular}
\end{table}
\end{center}

%
%
%
%

The numbers $u$ and $d$ of upper and lower boxes in $\lambda /\mu $ play
rather different roles in formula (\ref{eq:Lud}). However it turns out that
there is a remarkable hidden symmetry; to describe this we set%
\[
\tilde{L}\left( u,d\right) :=\frac{L\left( u,d;m,y\right) }{%
\prod_{i=1}^{d-m}(y+m+i)(y+i-1+2r)}.
\]

\begin{thm}
\label{thm:udsym} If $u,d$ are nonengative integers both $\geq m$, then we
have 
\[
\tilde{L}(u,d)=\tilde{L}(d,u).
\]
\end{thm}

The rest of this paper is arranged as follows. In Section \ref%
{sec:background} we briefly recall some background to partitions, Jack
polynomials and their binomial coefficients. Section \ref{sec:stemleaf}
contains a detailed description of the decomposition of $\lambda $ induced
by $\mu $ along with the stem-leaf factorization of $b_{\mu }^{\lambda }$.
Here we define the stem combinatorially (Definition \ref{def:stemleaf}) and
establish that the leaf depends only on $\lambda /\mu $ (Theorem \ref%
{prop:leafskew}). The short Section \ref{sec:crithook} serves to define
critical boxes and hooks which play an especially important role when $%
\lambda /\mu $ is a horizontal strip, i.e., when its constituent rows do not
overlap along any columns. In Section \ref{sec:horz} we study this case in
detail, and establish a combinatorial leaf formula which depends on the
critical hook (Theorem \ref{thm:q}). Section \ref{sec:overlap} is devoted to
the case where the two skew rows overlap along $m>0$ columns, and in this
case $m$ plays an important role in the resulting leaf formula (Theorem \ref%
{thm:p}). In each case, we also establish the symmetry between $u$ and $d$. 

\section{Background} \label{sec:background}

In this short section we collect together some
notation and background material for our paper. This has occasioned a slight
overlap with the introduction, especially with the concepts that were needed
for the formulation of the results. We fix throughout a natural number $n$
and a parameter $\alpha ,$ we write $\mathbb{F}=\mathbb{Q}\left( \alpha
\right) $ and put $r=1/\alpha$.

A \emph{partition} $\lambda $ of length $\leq M$ is a finite sequence of
non-negative integers 
\begin{equation*}
\lambda _{1}\geq \lambda _{2}\geq \cdots \geq \lambda _{M}\geq 0.
\end{equation*}%
The $\lambda _{i}$ are referred to as parts of $\lambda $, and we denote
their sum by 
\begin{equation*}
\left\vert \lambda \right\vert =\lambda _{1}+\cdots +\lambda _{M}
\end{equation*}%
We will identify a partition $\lambda $ with its \emph{Young diagram,} which
is a left justified array of boxes, with  $\lambda _{i}$ boxes in row $i$.
For instance, the diagram of $(4,3,2,1)$ (and also of $\left(
4,3,2,1,0\right) $ etc.) is as follows: 
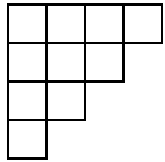
\begin{equation*}
\ytableaushort{~~~~,~~~,~~,~}
\end{equation*}

There are two natural partial orders on the set of partitions. The \emph{%
dominance} order $\mu \leq \lambda $ is defined by the requirement%
\begin{equation*}
\mu _{1}+\cdots +\mu _{i}\leq \lambda _{1}+\cdots +\lambda _{i}\text{ for
all }i.
\end{equation*}%
with equality for $i=M,$ i.e. we have $\left\vert \lambda \right\vert
=\left\vert \mu \right\vert .$ The containment order $\mu \subset \lambda $
is defined by%
\begin{equation*}
\mu _{i}\leq \lambda _{i}\text{ for all }i.
\end{equation*}%
In terms of Young diagrams, dominance means that $\mu $ is obtained from $%
\lambda $ by moving some boxes to lower rows; while containment means that
the diagram of $\lambda $ contains that of $\mu $. For $\mu \subset \lambda $%
, the \emph{skew diagram} $\lambda /\mu $ consists of boxes of $\lambda $
that are not in $\mu $. If $\lambda /\mu $ consists of exactly one box, we
say that $\lambda $ and $\mu $ are adjacent, and we write $\mu \subset
:\lambda $. 

For a box $s$ in $\lambda $, we write $\text{arm}_{\lambda }(s)$ for the
number of boxes to its right, and $\text{leg}_{\lambda }(s)$ for the number
of boxes below it. We define the \emph{lower }and  \emph{upper }$r$\emph{%
-hook} of $s$ to be 
\begin{equation*}
c_{\lambda }(s)=\text{arm}(s)+r\cdot \text{leg}\left( s\right) +r,\quad
c_{\lambda }^{\prime }(s)=\text{arm}(s)+r\cdot \text{leg}\left( s\right) +1.
\end{equation*}

\subsection{Jack Polynomials}
Jack polynomials form a linear basis for the algebra of symmetric
polynomials $\mathbb{F}\left[ x_{1},\ldots ,x_{M}\right] ^{S_{M}}$. They
arise as eigenfunctions of the Laplace-Beltrami operator  
\[
D(\al) = \frac{\al}{2} \sum_i x_i^2 \frac{\partial^2}{\partial x_i^2} + \sum_{i \neq j} \frac{x_i^2}{x_i-x_j} \frac{\partial}{\partial x_i}.
\]
It is readily checked that $D$ is upper-triangular with respect to the
dominance order on monomial symmetric functions $m_{\lambda }$ in the sense
that 
\[
D(\al) m_\la = \sum_{\mu\leq \la} c_{\la,\mu} m_\mu.
\]
$P_{\lambda }=P_{\lambda }^{(\alpha )}(x_{1},\ldots
,x_{M})$ is the unique eigenfunction of $D$ of the form 
\begin{equation*}
P_\la=m_\la+\sum_{\mu\lneq \la} v_{\la,\mu} m_\mu.
\end{equation*} 

We refer the reader to \cite{jack:69,knop:sahi:97,macdonald:95,stanley:89}
for more properties of Jack polynomials, including their role in algebraic combinatorics.

\subsection{Binomial Coefficients} \label{sec:recur}

For partitions $\mu \subset \lambda $, the \textit{Jack binomial coefficients%
} $b_{\mu }^{\lambda }$ are defined by the expansion 
\[
\frac{\Ja_\la(1+x_1,\ldots,1+x_M)}{\Ja_\la(1,\ldots,1)}=\sum_\mu b^\la_\mu \frac{\Ja_\mu(x_1,\ldots,x_M)}{\Ja_\mu(1,\ldots,1)}.
\]

For adjacent partitions $\mu \subset :\lambda $, one knows by \cite[Prop. 2]%
{kaneko:93}  that $b_{\mu }^{\lambda }$ equals the quantity 
\begin{align}
\label{a_lamu}
a^\lambda_\mu =\left(\prod_{s \in C} \frac{c_\lambda(s)}{c_\mu(s)} \right) \left( \prod_{s \in R} \frac{c'_\lambda(s)}{c'_\mu(s)} \right),
\end{align}
where $C$ and $R$ are the boxes in $\mu $ that lie in the same column and
row, respectively, as the unique box in $\lambda /\mu$. 

In \cite[Theorem 2]{sahi:11} one of us established the following recursion for binomial
coefficients:
\begin{align*}
\left(|\la|-|\mu|\right) \cdot b^\la_\mu =  \sum_{\ka\subset:\la} a^\la_\ka \cdot b^{\ka}_\mu.
\end{align*}
This can be reformulated as follows. For $\mu \subset \lambda 
$ let $\mathcal{T=T}^\la_\mu$ be the set of all ascending sequences $%
T$ of the form 
$$\mu = \mu_0 \subset: \mu_{1} \subset: \cdots \subset: \mu_n = \la.$$
For $T$ in $\mathcal{T}$ let $a_{T}$ denote the product of adjacent
coefficients 
\begin{align}
\label{a_T}
a_T = \prod^{n}_{i=1} a_{\mu_{i-1}}^{\mu_i}.
\end{align}
Then, as noted in the proof of \cite[Theorem 5]{sahi:11}, one has 
\begin{align}
\label{b_lamu}
n! \cdot b^\la_\mu =  \sum_T a_T.
\end{align}


\section{Stems and Leaves} \label{sec:stemleaf}

We now generalize the sets $C$ and $R$ from (\ref{a_lamu}), as foreshadowed in the introduction.
\begin{defn}
\label{inddecomp}
Given partitions $\la \supset \mu$, we decompose the Young diagram of $\la$ relative to $\mu$ into the following sets:
\begin{align*}
S\lm&=\la/\mu=\set{(i,j) \in \la \mid (i,j)\notin \mu} \mbox{ (\emph{skew} boxes)} \\
R\lm&=\set{(i,j) \in \mu \mid \mu_i < \la_i \text{ and } \mu'_j = \la'_j} \mbox{ (\emph{row} boxes)}, \\
C\lm&=\set{(i,j) \in \mu \mid \mu_i = \la_i \text{ and } \mu'_j < \la'_j} \mbox{ (\emph{column} boxes)}, \\
J\lm&=\set{(i,j) \in \mu \mid \mu_i < \la_i \text{ and } \mu'_j < \la'_j} \mbox{ (\emph{joint} boxes)}, \\
N\lm&=\set{(i,j) \in \mu \mid \mu_i = \la_i \text{ and } \mu'_j = \la'_j} \mbox{ (\emph{neutral} boxes)}.
\end{align*}
This collection of sets is called the \textbf{decomposition of $\la$ induced by $\mu$}, and the boxes in the diagram of $\la$ can be written as the disjoint union $S \cup R \cup C \cup D \cup N$. (We omit the indexing partitions $\la$ and $\mu$ whenever they are clear from context.) 
\end{defn}

\begin{expl} For $\la=(8,7,3,3,1)$ and $\mu=(8,4,3,1,1)$,
\[
\ytableaushort{NCCNCCCN,RJJRSSS,NCC,RSS,N}*[*(lightgray)]{0,4+3,0,1+2}
\]
\end{expl}


Note that $J$ is determined completely by the skew diagram $S$, since its boxes are precisely those which share both a row and column with some box in $S$. We refer to $S_+ = S\cup J$ as the \emph{completion} of $S$. 

\begin{defn} 
\label{def:stemleaf}
For $\lambda \supset \mu$ the \textbf{stem} is defined as follows: 
\[
K^\la_\mu= \frac{\displaystyle\left(\prod_{b \in C} \frac{c_\lambda(b)}{c_\mu(b)}\right) \left(\prod_{b \in R} \frac{c'_\lambda(b)}{c'_\mu(b)}\right)}{\displaystyle\left(\prod_{b \in J}c_\mu(b)c'_\mu(b)\right)},
\] and the \textbf{leaf} is defined to be the quotient $\displaystyle L^\la_\mu=\frac{b^\la_\mu}{K^\la_\mu}.$ 
\end{defn}

We are now ready to prove our first main theorem.

\begin{thmn} The leaf $L^\la_\mu$ depends only on the skew diagram $\la/\mu$. 
\end{thmn}

\begin{proof}  
Letting $n$ equal $|\la|-|\mu|$, recall from (\ref{b_lamu}) that  $n!\cdot b^\lambda_\mu$ is the sum of multiplicities $a_{T}$ over all ascending sequences $T$ of the form $\mu = \mu_0 \subset: \cdots \subset: \mu_n = \la$. For each such $T$, we denote by $C^T_i$ and $R^T_i$ those boxes in $\mu_{i-1}$ which lie in the same column and row as the unique box of $\mu_{i}/\mu_{i-1}$. Using (\ref{a_T}) and then (\ref{a_lamu}), we obtain
\begin{align*}
a_{T} &= \prod^{n}_{i=1} a_{\mu_{i-1}}^{\mu_i}= \prod^{n}_{i=1} \left(\prod_{s \in C^T_i} \frac{c_{\mu_{i}}(s)}{c_{\mu_{i-1}}(s)} \right) \left( \prod_{s \in R^T_i} \frac{c'_{\mu_{i}}(s)}{c'_{\mu_{i-1}}(s)} \right).
\end{align*}
We observe that $a_T$ may also be expressed as a product over all boxes of $\la$ 
\[
a_{T}=\prod_{s \in \la} \left( \prod^{n}_{i=1} h^T_i(s) \right),
\]
where
\begin{align}
\label{eq:h}
h^T_i(s)=\begin{cases} \frac{c_{\mu_{i}}(s)}{c_{\mu_{i-1}}(s)} &\mbox{ if } s \in C^T_i \\
																			\frac{c'_{\mu_{i}}(s)}{c'_{\mu_{i-1}}(s)} &\mbox{ if } s \in R^T_i\\
																			1 &\mbox{ otherwise}.
								\end{cases}
\end{align} 

Now recall that each box $s$ of $\la$ lies in precisely one of the five sets from Defintion \ref{def:stemleaf}. If $s\in N$, then $s \notin R^T_i  \cup C^T_i$ for any sequence $T$ and index $i$, so none of its hook ratios appear as factors in $a_T$. And if $s \in R$, then $s$ is not an element of of $C^T_i$ for any $i$, but the set $I$ of indices $i$ for which $s \in R^T_i$ is non-empty. If $j\notin I$, then $c'_{\mu_j}(s)$ and $c'_{\mu_{j-1}}(s)$ are equal, so $h^T_j(s)=1$. Therefore, in this case we obtain (via telescoping product)
\[ 
\prod^{n}_{i=1} h^T_i(s)=\prod_{i=1}^n \frac{c'_{\mu_{i}}(s)}{c'_{\mu_{i-1}}(s)}=\frac{c'_{\la}(s)}{c'_{\mu}(s)}.
\]
By an analogous argument, if $s \in C$, then 
\[
\prod^{n}_{i=1} h^T_i(s)=\prod_{i=1}^n \frac{c_{\mu_{i}}(s)}{c_{\mu_{i-1}}(s)}=\frac{c_{\la}(s)}{c_{\mu}(s)}.
\]
Absorbing the $C$, $R$ and $N$ boxes' contributions to the multiplicity of $T$ into the stem $K^\la_\mu$, we have
\[
a_{T}=K^\la_\mu \left(\prod_{s\in J} c_\mu(s)c'_\mu(s) \right) \left(\prod_{s \in S_+}  \prod^{n}_{i=1} h^T_i(s) \right),
\]
whence
\[
n!L^\la_\mu= \frac{n!b^\la_\mu}{K^\la_\mu}= \left(\prod_{s\in J} c_\mu(s)c'_\mu(s) \right) \sum_{T} \left( \prod_{s \in S_+} \prod^{n}_{i=1} h^T_i(s) \right).
\]
Since all indices in the rightmost expression above are sourced from the completion $S_+ = S \cup J$, it follows that the leaf depends only on $S = \la/\mu.$
\end{proof}

In light of Theorem \ref{prop:leafskew}, we will henceforth denote the leaf by $L_{\la/\mu}$ rather than $L^\la_\mu$. Leaves inherit the following recurrence from $b^\la_\mu.$ 

\begin{prop} Given partitions $\mu \subset \la$ with $|\la|-|\mu|=n$, the leaf $L_{\la/\mu}$ satisfies
\begin{align*}
nL_{\la/\mu}&=\sum_{\ka}L_{\ka/\mu}\prod_{s\in S_+} h^\la_\ka(s)\ell^\ka_\mu(s), 
\end{align*}
where the sum is over $\ka$ satisfying $\mu \subset \ka \subset: \la$, and where $h^\la_\ka(s)$ and $\ell^\ka_\mu(s)$ are defined by:
\begin{align*}
h_\star^\bullet(s)=\begin{cases}
\frac{c_\bullet(s)}{c_\star(s)} &\mbox{ if } s \in C_\star^\bullet \\
\frac{c'_\bullet(s)}{c'_\star(s)} &\mbox{ if } s \in R_\star^\bullet \\
1 &\mbox{ otherwise } \\
\end{cases}
\qquad 
\mbox{ and }
\qquad \ell_\star^\bullet(s)=\begin{cases}
c_\bullet(s)c'_\star(s) &\mbox{ if } s \in C_\star^\bullet \\
c'_\bullet(s)c_\star(s) &\mbox{ if } s \in R_\star^\bullet \\
1 &\mbox{ otherwise } \\
\end{cases}
\end{align*}

\label{prop:leafrec}
\end{prop}
\begin{proof}
This recurrence follows from (the proof of) Theorem \ref{prop:leafskew} by decomposing the set of all ascending sequences $T$ of the form $\mu \subset: \cdots \subset: \la$ into a sequence $\mu \subset: \cdots \subset: \ka$ of length $(n-1)$ followed by a sequence $\ka \subset: \la$ of length $1$ .
\end{proof}

\begin{rem}
\label{rem:leafrec}
If we decompose $T$ into a sequence $\mu\subset:\nu$ of length $1$ followed by a sequence $\nu\subset: \cdots \subset: \la$ of length $(n-1)$, we obtain a new dual recurrence
\[
nrL_{\la/\mu}=\sum_{\nu}L_{\la/\nu}\prod_{s\in S_+} h^\nu_\mu(s)\ell^\la_\nu(s)\prod_{s\in J} \frac{c_\mu(s)c'_\mu(s)}{c_\nu(s)c'_\nu(s)},
\]
where the sum is now indexed by all partitions $\nu$ satisfying $\mu \subset: \nu \subset \la$. 
\end{rem}

\section{Critical Boxes and Hooks} \label{sec:crithook}

For the remainder of this paper, we fix partitions $\mu \subset \la$ with $|\la|-|\mu|=n$. Letting $N, S, R, C$ and $J$ be the sets from Definition \ref{inddecomp}, we will henceforth assume that skew diagram $S=\la/\mu$ consists of at most two rows $S_1$ and $S_2$. We call $S$ a {\em horizontal strip} if $S_1$ and $S_2$ do not share any columns. Let $u$ be the number of boxes in $S_1$ (the first row) and $d=n-u$ be the number of boxes in $S_2$ (the second row). By convention, if $S$ consists of a single row, then $S_2$ is empty and we have $u=n$ and $d=0$.

Note that $J$ is entirely contained within a single row, and when $S$ is a horizontal strip $J$ contains exactly $d$ boxes. On the other hand, if $S$ is not a horizontal strip, then $S_1$ and $S_2$ overlap in $m \geq 1$ columns and $J$ contains $d-m$ boxes. 

The rightmost box of $J$ plays an especially important role in our calculations. We note that its coordinates can be completely determined from the skew diagram $S$ (since it has the same row coordinates as the boxes in $S_1$ and the same column coordinate of the rightmost box in $S_2$).

\begin{defn} 
\label{def:crit}
The \textbf{critical box} of $S=\la/\mu$ is the rightmost box $x^{\ast}$ of the set $J$. The \textbf{critical hook} of $S=\la/\mu$ is the polynomial in $r$ given by
\[ 
y = y^\la_\mu(r) = \text{arm}_\mu(x^*)+r\cdot \text{leg}_\mu(x^*).
\]
\end{defn}

Note that the critical hook is not itself an upper or lower hook, but rather obtained from these hooks by removing the corner box.

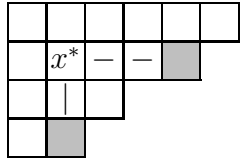
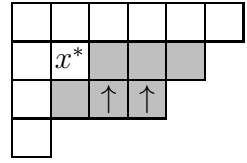
\begin{center}
\begin{tabular}{lcl}
$\ytableaushort{~~~~~~,~{x^{*}}--~,~{|}~,~~}*[*(lightgray)]{0,4+1,0,1+1}$ &\qquad \qquad \qquad \qquad \qquad & 
$\ytableaushort{~~~~~~,~{x^{*}}~~~,~~\uparrow\uparrow,~}*[*(lightgray)]{0,2+3,1+3,0}$\\
$y = 2+r$ && $y =0$ \\
$m=0$ && $m=2$
\end{tabular}
\end{center}


We observe that if $m>0$, then $S_1$ and $S_2$ must be in successive rows. Therefore, $y$ is nonzero only if $m$ is zero, which in turn occurs if and only if $S$ is a horizontal strip. We therefore deal with the cases $m=0$ (the \emph{gap} case) and $m>0$ (the \emph{overlap} case) separately in the next two sections. In each case, our strategy is to find a suitable recurrence relation for the leaf $L_{\la/\mu}$ and to show that our formula from Theorem \ref{thm:main} satisfies this recurrence. 

\section{The Gap Case} \label{sec:horz}

Our goal in this section is to prove the following result.
\begin{thm} 
\label{thm:q}
Let $\mu \subset \la$ be partitions with $|\lambda|-|\mu| = n$ such that the skew diagram $\lambda/\mu$ is a horizontal strip consisting of two rows. The leaf $L_{\lambda/\mu}$ is given by
\[
L_{\lambda/\mu}=\sum^d_{k=0} \binom{d}{k}\prod_{i=0}^{k-1} (i+1-r)(i+r)\prod_{i=k+1}^d(y+d+r-i)(y+u+r+i),
\]
where $u$ and $d$ denote the number of boxes in the upper and lower rows of $\la/\mu$ respectively, while $y$ is the associated critical hook from Definition \ref{def:crit}.
\end{thm}

Under the assumptions of this theorem, the set $J$ contains $d$ boxes $x_1,\ldots,x_d$ in a single row and the critical box $x^*$ is $x_d$. Since the number of overlapping columns $m$ between $S_1$ and $S_2$ is zero, the skew diagram $\la/\mu$ is determined by $u,d,$ and the critical hook $y = y^\la_\mu(r)$. It will thus be convenient to denote the leaf $L_{\la/\mu}$ by $\Q^u_d(y)$ throughout this section.

\begin{prop} 
The leaf $\Q^u_d(y)$ satisfies the recurrence 
\begin{align*}
\alpha_0 \cdot \Q^u_d(y) = \alpha_1 \cdot \Q^{u-1}_{d}(y) + \alpha_2 \cdot \Q^{u}_{d-1}(y+1),
\end{align*}
where the polynomials $\alpha_\bullet = \alpha_\bullet(y,u,d,r)$ are given by 
\begin{align*}
\alpha_0(y,u,d,r) & = (u+d)(y+u+r), \\
\alpha_1(y,u,d,r) &= u(y+{u+d}+r), \text{ and}  \\
\alpha_2(y,u,d,r) &= d(y+r)(y+u+1)(y+u+2r),
\end{align*}
along with the initial condition $\Q^0_0(y)=1$.
\label{prop:leaf1}
\end{prop}

\begin{proof} The initial condition follows immediately from observing that when $n=u+d=0$, we have $\la=\mu$, and the sets $R$, $C$ and $J$ are empty. In this case, the binomial coefficient $b^\la_\mu$ and its stem $K^\la_\mu$ both trivially equal $1$, so $L_{\la/\mu} = Q_0^0(y) = 1$. 

Now suppose $n>0$, and let $\ka$ and $\nu$ be the partitions whose Young diagrams are obtained from $\la$ by removing the rightmost boxes of $S_1$ and $S_2$ respectively. Recall the recursive expression for the leaf from Proposition \ref{prop:leafrec}:
\begin{align}
\label{eq:L_lamu_recur}
n L_{\la/\mu} &= L_{\ka/\mu}\prod_{s\in S_+} h^\la_{\ka}(s)\ell^{\ka}_\mu(s) + L_{\nu/\mu}\prod_{s\in S_+} h^\la_\nu(s)\ell^\nu_\mu(s).
\end{align}
Since $\ka/\mu$ has the same critical hook as $\la/\mu$ but one fewer box in the upper row, we obtain $L_{\ka/\mu}=\Q^{u-1}_d(y)$. On the other hand, $\nu/\mu$ has one fewer box in the lower row, so its critical hook is $y+1$, and we have $L_{\nu/\mu}=\Q^{u}_{d-1}(y+1)$. We conclude the argument by establishing  
\[
\frac{\alpha_2(y,u,d,r)}{\alpha_0(y,u,d,r)} = \frac{1}{n}\prod_{s\in S_+} h^\la_\nu(s)\ell^\nu_\mu(s),
\] and leave the (easier) verification involving $\alpha_1/\alpha_0$ as an exercise. To this end, note that 
\[
S_+ \cap C^\la_\nu  = \{x_d\} \text{ and } S_+ \cap R^\la_{\nu} = S_2 \cap \nu, 
\] so we have an expression for the $h$-product
\[
\prod_{s\in S_+} h^{\la}_{\nu}(s) =  d\frac{c_\la(x_d)}{c_{\nu}(x_d)}=\frac{d(y+u+2r)}{(y+u+r)}.
\]
And similarly, 
\[
S_+ \cap C^{\nu}_\mu =  \varnothing  \text{ and } S_+ \cap R^{\nu}_\mu = \{x_d\},
\] which yields an expression for the $\ell$-product
\[
\prod_{s\in S_+} \ell^{\nu}_{\mu}(s) = c_{\mu}(x_d)c'_\nu(x_d)=(y+u+1)(y+r).
\]
The desired result now follows from multiplying the $h$ and $\ell$-products described above, and recalling that $n = u+d$.
\end{proof}

\subsection{Proof of Theorem \ref{thm:q}}

We will describe a general family of polynomials indexed by $u$ and $d$ and show that they can be suitably modified to satisfy the same recurrence as $\Q^u_d(y)$ from Proposition \ref{prop:leaf1}.

\begin{defn} Let $\M^u_d$ be the bivariate polynomial given by
\[
\M^u_d(z;\theta)=\sum^d_{k=0} \binom{d}{k}  \prod^{k-1}_{i=0} (\theta+\rho_i^2)\prod_{i=k+1}^d(z-i)(z+i+u-d),
\]
where $\rho_i$ equals $i+\frac{1}{2}.$
\label{defn:M}
\end{defn}

We now claim the following.

\begin{prop} 
The polynomials $\M^u_d$ satisfy the recurrence
\[
\alpha'_0\cdot \M^u_d = \alpha'_1\cdot \M^{u-1}_{d}+ \alpha'_2\cdot \M^{u}_{d-1},
\]
where the polynomials $\alpha'_\bullet = \alpha'_\bullet(u,d,z,\theta)$ are given by 
\begin{align*}
\alpha'_0(u,d,z,\theta) &= (u+d)(z+u-d), \\
\alpha'_1(u,d,z,\theta) &= u(z+u),\text{ and} \\
\alpha'_2(u,d,z,\theta) &= d(z-d)(\theta + (z+\rho_{u-d})^2),
\end{align*}
along with the initial condition $\M^0_0=1$.
\label{prop:M}
\end{prop}

This recurrence given in Proposition \ref{prop:M} can be modified to the recurrence satisfied by $\Q^u_d$ in Theorem \ref{thm:q} by making the following substitutions:
\begin{align*}
 z &= y+d+r, \text{ and} \\ 
\theta &= -\left(r-\frac{1}{2}\right)^2.
\end{align*} 
Therefore, in order to prove Theorem \ref{thm:q}, it suffices to prove Proposition \ref{prop:M}. To accomplish this, we first establish a different recurrence for $\M^u_d$.

\begin{lem} The polynomials $\M^u_d$ satisfy the recurrence
\begin{equation}
\M^u_d=\left(\theta + (z+\rho_{u-d})^2\right)\M^u_{d-1}-u(z+u)\M^{u-1}_{d-1},
\label{eq:M}
\end{equation}
where $\M^u_{0}=1$ for all $u$.
\label{lem:M}
\end{lem}

\begin{proof} For brevity, we define $q_k(\theta) = \prod^{k-1}_{i=0} (\theta+\rho_i^2)$ and 
\[
f^u_d(k) = \binom{d}{k} \prod_{i=k+1}^d(z-i)(z+i+u-d), 
\]
so $\M^u_d=\sum^d_{k=0} f^u_d(k)q_k(\theta)$. Note that the $q_k(\theta)$ form a basis for polynomials in $\theta$ over the field $\mathbb{Q}(z,u,d).$ In fact, we will find it convenient to suppress $\theta$, and write $q_k=q_k(\theta)$. It now suffices to show that the coefficient of $q_k$ on the right side of (\ref{eq:M}) also equals $f^u_d(k)$. 

Expanding in terms of the $f$'s and $q$'s, the right side of (\ref{eq:M}) equals
\begin{equation}
\left(\theta + (z+\rho_{u-d})^2\right)\sum_{k=0}^{d-1} f^{u}_{d-1}(k)q_k-u(z+u)\sum_{k=0}^{d-1} f^{u-1}_{d-1}(k)q_k.
\label{eq:rM1}
\end{equation}
We can absorb the expression $\theta + (z+\rho_{u-d})^2$ outside the first sum into the sum, and then for each $k$, we may rewrite this expression as 
\[
(\theta+\rho_k^2) + \left((z+\rho_{u-d})^2-\rho_k^2\right).  
\]
With this modification, (\ref{eq:rM1}) becomes
\begin{equation}
\sum^{d-1}_{k=0} f^{u}_{d-1}(k)q_{k+1} +\sum^{d-1}_{k=0}\left( (z+\rho_{u-d})^2-\rho_{k}^2\right)f^{u}_{d-1}(k)q_k-\sum^{d-1}_{k=0} u(z+u)f^{u-1}_{d-1}(k)q_k.
\label{eq:rM2}
\end{equation}
Collecting the coefficients of $q_k$ in (\ref{eq:rM2}), we have
\begin{equation}
f^{u}_{d-1}(k-1)+\left( (z+\rho_{u-d})^2-\rho_{k}^2\right) f^{u}_{d-1}(l)-u(z+u)f^{u-1}_{d-1}(k).
\label{eq:rM3}
\end{equation}
In order to establish (\ref{eq:M}), we must now show that expression (\ref{eq:rM3}) equals $f^u_d(k).$

Elementary calculations express each term in (\ref{eq:rM3}) as a multiple of $f^u_d(k)$:
\begin{align*}
f^{u}_{d-1}(k-1) &= \frac{k(z-k)}{d(z-d)} \cdot f^{u}_d(k), \\
\left((z+\rho_{u-d})^2-\rho_{k}^2\right) \cdot f^{u}_{d-1}(k) &= \frac{(d-k)(z+u-d-k)}{d(z-d)}\cdot f^{u}_d(k), \\
u(z+u) \cdot f^{u-1}_{d-1}(k) &= \frac{u(d-k)}{d(z-d)} \cdot f^{u}_d(k).
\end{align*}
(Note that the second calculation requires factoring the expression $(z+\rho_{u-d})^2-\rho_k^2$ as $(z+\rho_{u-d}+\rho_k)(z+\rho_{u-d}-\rho_k)$.) Subtracting the third quantity from the sum of the first two now yields $f^u_d(k)$ on the right side of (\ref{eq:M}) and hence concludes the argument.
\end{proof}

Although the recurrence of Lemma \ref{lem:M} looks quite different from the one in Proposition \ref{prop:M} at first glance, it can now be used to yield the desired proof.
%

\begin{proof}[Proof of Proposition \ref{prop:M}.] We expand the difference $D = \alpha'_0\cdot \M^u_d - \alpha'_1\cdot \M^{u-1}_d$, using the notation $q_k$ and $f^u_d(k)$ introduced the proof of Lemma \ref{lem:M}, to get
\[ 
D = \sum_{k=0}^{d} \left((u+d)(z+u-d) f^{u}_d(k) q_k- u(z+u)f^{u-1}_d(k) q_k\right)
\]
It is readily checked that $(z+u)f^{u-1}_d(k)$ equals $(z+k+u-d)f^u_d(k)$, so we obtain
\begin{align*}
D &= \sum_{k=0}^{d} \big( (u+d)(z+u-d)-u(z+k+u-d) \big) f^{u}_d(k) q_k \\
	  &= d(z-d)\M^u_d + \sum_{k=0}^{d-1} u(d-k) f^{u}_d(k) q_k.
\end{align*}
By the third elementary calculation mentioned at the end of the proof of Lemma \ref{lem:M}, we know $\frac{u(d-k)}{d(z-d)} \cdot f^{u}_d(k)$ equals $u(z+u) \cdot f^{u-1}_{d-1}(k)$, so we obtain
\begin{align*}
D &= d(z-d)\M^u_d + \sum_{k=0}^{d-1} d(z-d)u(z+u) f^{u-1}_{d-1}(k) q_k \\
		&= d(z-d)\left[\M^u_d + u(z+u)\M^{u-1}_{d-1}\right].
\end{align*}
Applying Lemma \ref{lem:M} yields $D = \alpha'_2\cdot \M^u_{d-1}$, as desired. 
\end{proof}
%

\subsection{Symmetry between $u$ and $d$}

We now consider some algebraic identities related to the polynomial expression for the leaf $L^\lambda_\mu$ from Theorem \ref{thm:q}, i.e.,
\[
\Q^u_d(y) = \sum^d_{k=0} \binom{d}{k}\prod_{i=0}^{k-1} (i+1-r)(i+r)\prod_{i=k+1}^d(y+d+r-i)(y+u+r+i).
\]


The parameters $u$ and $d$ appear to be playing remarkably different roles in the recursion of Proposition \ref{prop:leafrec} and in our explicit formula from Theorem \ref{thm:q}. However, we find a surprising symmetry between $u$ and $d$ involving the polynomials 
\[
\phi_k(y)=\displaystyle \prod_{i=1}^k (y+i)(y+i-1+2r).
\]
We observe that setting $u=0$ in the recurrence of Proposition \ref{prop:leaf1} yields 
\[
\Q^0_d(y) = (y+1)(y+2r)\Q^{0}_{d-1}(y+1), \text{ and }\Q^0_0(y) = 1.
\]
Since $\phi_d$ also satisfies this recurrence, it follows that $\phi_d(y) = \Q^0_d(y)$. Moreover, the following $\phi$-identities are easy to verify
\begin{align}
\phi_d(y) &=(y+1)(y+2r) \cdot  \phi_{d-1}(y+1),  \label{eq:phid1}\\
\phi_d(y) &=(y+d)(y+d+1+2r) \cdot  \phi_{d-1}(y),  \label{eq:phid2} \text{ and}\\
\phi_d(y+1) &= \frac{(y+d+1)(y+d+2r)}{(y+1)(y+2r)} \cdot \phi_d(y). \label{eq:phid3}
\end{align}

\begin{thm} 
The polynomials $\Q^u_d$ and $\phi_d$ satisfy $\Q^u_d(y)\cdot \phi_u(y)=\Q^d_u(y)\cdot \phi_d(y)$.
\label{thm:symQ}
\end{thm}

\begin{proof}
We first observe that a new recurrence for $\Q^u_d(y)$ may be derived by modifying the proof of Proposition \ref{prop:leaf1} as follows. Instead of constructing the intermediate partitions $\kappa$ and $\nu$ by removing a box each from $\lambda$, we produce $\kappa$ by adding the leftmost box of the top row $S_1$ to $\mu$, and produce $\nu$ by adding the rightmost box of $S_2$ to $\mu$. And rather than using the recurrence from Proposition \ref{prop:leafrec}, we employ the dual recurrence of Remark \ref{rem:leafrec}. With these alterations in place, and using precisely the same strategy as before, one obtains
\[
\bar{\alpha}_0 \cdot \Q^u_d(y) = \bar{\alpha}_1 \cdot \Q^{u-1}_{d}(y+1) + \bar{\alpha}_2 \cdot \Q^{u}_{d-1}(y),
\]
with
\begin{align*}
\bar{\alpha}_0(y,u,d,r) &= (u+d)(y+d+r),\\
\bar{\alpha}_1(y,u,d,r) &= u(y+r)  \text{, and} \\
\bar{\alpha}_2(y,u,d,r) &= d(y+u+d+r)(y+d-1+2r)(y+d),
\end{align*}
along with the initial condition $\Q^0_0(y)=1$. 

We now use this dual recurrence to prove the desired symmetry. We proceed by induction on $n=u+d$, noting that the result holds trivially for $u=d=0$. For $n > 0$, consider our new recurrence with the roles of $u$ and $d$ interchanged:
\[
\bar{\alpha}_0(y,d,u,r)\cdot \Q^d_u(y) = \bar{\alpha}_1(y,d,u,r)\cdot\Q^{d-1}_{u}(y+1) + \bar{\alpha}_2(y,d,u,r)\cdot\Q^{d}_{u-1}(y).
\]
Multiplying throughout by $\phi_d(y)$ and using identity (\ref{eq:phid1}) with the first term on the right side, we note that $\bar{\alpha}_0(y,d,u,r)\cdot\Q^d_u(y)\phi_d(y)$ equals  
\[
\bar{\alpha}_1(y,d,u,r) (y+1)(y+2r)\cdot\Q^{d-1}_{u}(y+1)\phi_{d-1}(y+1) + \bar{\alpha}_2(y,d,u,r) \cdot \Q^{d}_{u-1}(y)\phi_d(y).
\]
And applying the inductive hypothesis to the expression above yields
\[
\bar{\alpha}_1(y,d,u,r)(y+1)(y+2r)\cdot\Q^{u}_{d-1}(y+1)\phi_{u}(y+1) + \bar{\alpha}_2(y,d,u,r)\cdot\Q^{u-1}_{d}(y)\phi_{u-1}(y).
\]
Identities (\ref{eq:phid2}) and (\ref{eq:phid3}) may now be used to express both summands above in terms of $\phi_u(y)$. Using these identities and the definition of the $\bar{\alpha}_\bullet$'s, our expression for $\bar{\alpha}_0(y,d,u,r) \cdot \Q^d_u\phi_d(y)$ simplifies to
\[
\left[d(y+r)(y+u+1)(y+u+2r)\cdot\Q^{u}_{d-1}(y+1) + u(y+u+d+r)\cdot\Q^{u-1}_{d}(y) \right]\phi_{u}(y).
\]
By Proposition \ref{prop:leaf1}, the factor within square brackets equals $\alpha_0(y,u,d,r)\cdot \Q^u_d(y)$. But since 
\[
\alpha_0(y,u,d,r) = (u+d)(y+u+r) = \bar{\alpha}_0(y,d,u,r), 
\]
the desired result follows.
\end{proof}

\section{The Overlap Case} \label{sec:overlap}

 This section is devoted to proving the following result.
\begin{thm} 
\label{thm:p}
Let $\mu \subset \la$ be partitions with $|\lambda|-|\mu| = n$ so that the skew diagram $\lambda/\mu$ consists of two rows which overlap along $m \geq 1$ columns. Then, the leaf $L_{\lambda/\mu}$ is given by
\[
L_{\lambda/\mu}=\sum^{d-m}_{k=0} \binom{d-m}{k}\prod_{i=0}^{k-1} (m+i+1-r)(i+r)\prod_{j=k+1}^{d-m}(d-m+r-j)(u+r+j),
\]
where $u$ and $d$ denote the number of boxes in the upper and lower rows of $\la/\mu$ respectively.
\end{thm}

When the rows $S_1$ and $S_2$ of $S$ overlap in $m$ columns, the set $J$ consists of $d-m$ boxes $x_1,\ldots,x_{d-m}$ (from left to right) in a single row. Recall that in this case the critical hook $y$ equals $0$, so $L^\la_\mu$ depends only $u$, $d$ and $m$, and since $m$ denotes the number of overlapping columns, we have $m\leq \min\{u,d\}$. It will be convenient to denote $L^\la_\mu$ by $\P^u_d(m)$ in this case. We can then characterize the leaf $\P^u_d(m)$ using the following recurrence.

\begin{prop} The leaf $\P^u_d(m)$ satisfies the recurrence relation
\[
\beta_0 \cdot \P^u_d(m) = \beta_1 \cdot \P^{u-1}_{d}(m) + \beta_2 \cdot \P^{u}_{d-1}(m-1),
\]
where
\begin{align*}
\beta_0(m,u,d,r) &= (u+d)(u-m+r), \\
\beta_1(m,u,d,r) &= (u-m)(u+d-m+r), \text{ and}\\
\beta_2(m,u,d,r) &= d(u-m+2r), 
\end{align*}
with the initial condition $\P^u_d(0)=\Q^u_d(0)$, where $\Q^u_d$ denotes the polynomials from Theorem \ref{thm:q}.
\label{prop:leaf3}
\end{prop}
\begin{proof} The argument is entirely similar to the one employed in the proof of Proposition \ref{prop:leaf1}. 
\end{proof}

\subsection{Proof of Theorem \ref{thm:p}}

We provisionally define
\begin{align}
\label{eq:Nshift}
\N^u_d(m) = \sum^{d}_{k=0} \binom{d}{k}\prod_{i=0}^{k-1} (m+i+1-r)(i+r)\prod_{j=k+1}^{d}(d-j+r)(u+j+r),
\end{align}
and wish to establish that (up to a simple modification $d \mapsto d-m$) these polynomials satisfy the same recurrence as the one described in Proposition \ref{prop:leaf3} for $\P^u_d$. For brevity, we will henceforth write
\begin{align}
\label{eq:Nbasis}
\N^u_{d}(m) = \sum_{k=0}^{d}g_k(u,d)p_k(m),
\end{align}
where 
\begin{align}
p_k(m) &= \prod_{i=0}^{k-1}(m+i+1-r), \text{ and} \label{eq:p}\\
g_k(u,d) &= \binom{d}{k}\prod_{i=0}^{k-1} (r+i)\prod_{j=k+1}^{d}(d-j+r)(u+j+r).  \label{eq:f}
\end{align}

We proceed by recalling that the recurrence from Proposition \ref{prop:leaf3} expresses $\P^u_d(m)$ in terms of $\P^{u-1}_d(m)$ and $\P^u_{d-1}(m-1)$. The upcoming Propositions \ref{prop:T1}, \ref{prop:T2} and \ref{prop:T3} consider corresponding expressions involving $\N^u_d(m)$, and show that the first equals the sum of the next two when expanded in terms of the basis $p_k(m)$ of polynomials in $m$ over the field $\mathbb{Q}(r,u,d)$. In each argument it becomes necessary to manipulate the $p_k(m)$ to pass between adjacent indices, and the following lemma allows us to perform such manipulations.

\begin{lem}
\label{lem:xshift}
For any polynomial $h = h(y,u,d,r)$, we have 
\begin{align}
(h + m) p_{k}(m) &=  \big[(h-k-1+r) p_{k}(m) + p_{k+1}(m) \big], \text{ and} \label{eq:h+}\\
(h - m)  p_{k}(m) &=  \big[(h+k+1-r) p_{k}(m) - p_{k+1}(m)\big]. \label{eq:h-}
\end{align}
\end{lem}
\begin{proof}
Both statements are derivable from the identity
\begin{align*}
p_{k+1}(m) &= (m-r+k+1) \cdot p_{k}(m),
\end{align*} 
which follows directly from (\ref{eq:p}). For instance, (\ref{eq:h+}) is verified by noting \[h + m = (h-k-1+r) + (m+ k+1-r).\] 
\end{proof}

\begin{prop}
\label{prop:T1}
Given $\N^u_d(m)$ as in (\ref{eq:Nbasis}), we have
\[
(u+d+m)(u-m+r)\N^u_d(m) = \sum_{k=0}^d \sum_{i=0}^2 g_k(u,d)  \gamma^{1,i}_{k}(u,d)p_{k+i}(m),
\]
where the three functions $\gamma^{1,\bullet}_k$ are given by
\begin{align*}
\gamma^{1,0}_k(u,d) &= (u+d-k-1+r)(u+k+1) \\
\gamma^{1,1}_k(u,d) &= 2k+3-d-r, \text{ and} \\
\gamma^{1,2}_k(u,d) &= -1.
\end{align*}
\end{prop}

\begin{proof}
Let $\LS$ denote the left side of the desired equality. Applying (\ref{eq:h-}) with $h = u+r$ gives
\[
\LS = (u+d+m)\left[\sum_{k=0}^d g_k(u,d) (u+k+1)p_k(m) - \sum_{k=0}^d g_k(u,d) p_{k+1}(m)\right].
\]
By (\ref{eq:h+}), with $h = u+d$, the first sum becomes
\begin{align*}
\LS_1 &= (u+d+m)\sum_{k=0}^d g_k(u,d) (u+k+1)p_k(m) \\
&= \sum_{k=0}^d g_k(u,d) \big[(u+d-k-1+r)(u+k+1)p_k(m) + (u+k+1)p_{k+1}(m)\big].
\end{align*}
Similarly, the second sum becomes
\begin{align*}
\LS_2 &= -(u+d+m)\sum_{k=0}^d g_k(u,d) p_{k+1}(m)\\
&= - \sum_{k=0}^d g_k(u,d) \big[ (u+d-k-2-r)p_{k+1}(m) + p_{k+2}(m) \big]
\end{align*}
Observing that $\LS = \LS_1 + \LS_2$ and collecting the coefficients of the $p_k(m)$ completes the argument.
\end{proof}

The proof of our second proposition uses the following elementary consequence of (\ref{eq:f}):
\begin{align}
\label{eq:g_ushift}
(u+d+r)\cdot g_k(u-1,d) = (u+k+r)\cdot g_k(u,d).
\end{align}

\begin{prop}
\label{prop:T2}
Given $\N^u_d(m)$ as in (\ref{eq:Nbasis}), we have
\[
(u-m)(u+d+r)\N^{u-1}_d(m) = \sum_{k=0}^d \sum_{i=0}^2 g_k(u,d)  \gamma^{2,i}_{k}(u,d)p_{k+i}(m),
\]
where the three functions $\gamma^{2,\bullet}_k$ are
\begin{align*}
\gamma^{2,0}_k(u,d) &= (u+k+1-r)(u+k+r) \\
\gamma^{2,1}_k(u,d) &= -(u+k+r), \text{ and} \\
\gamma^{2,2}_k(u,d) &= 0.
\end{align*}
\end{prop}
\begin{proof}
The left side $\LS$ is given by
\[
\LS = (u-m)(u+d+r) \sum_{k=0}^d g_k(u-1,d) p_k(m),
\]
and we use (\ref{eq:g_ushift}) to incorporate the $(u+d+r)$ factor within the sum. This gives
\[
\LS = (u-m)\sum_{k=0}^d (u+k+r) g_k(u,d) p_k(m),
\]
and it only remains to incorporate the $(u-m)$ factor. For this purpose, we once again apply (\ref{eq:h-}), this time with $h = u$, and obtain the desired result.
\end{proof}

Our third proposition requires a new identity which is easily verified via (\ref{eq:p}).
\begin{align}
p_k(m-1) &= p_k(m) - k \cdot p_{k-1}(m) \label{eq:p_mshift}
\end{align}

\begin{prop}
\label{prop:T3}
Given $\N^u_d(m)$ as in (\ref{eq:Nbasis}), we have
\[
(d+m)(u-m+2r)\N^{u}_{d}(m-1) = \sum_{k=0}^d \sum_{i=0}^2 g_k(u,d)  \gamma^{3,i}_{k}(u,d)p_{k+i}(m),
\]
where the three functions $\gamma^{3,\bullet}_k$ are given by
\begin{align*}
\gamma^{3,0}_k(u,d) &= (d-k+r)(u+k+r)-(k+1)(u-d+2k+1)-(d-k)(k+r), \\
\gamma^{3,1}_k(u,d) &= (u-d+3k+3), \text{ and} \\
\gamma^{3,2}_k(u,d) &= -1.
\end{align*}
\end{prop}
\begin{proof}
Using (\ref{eq:p_mshift}), we can decompose the left side $\LS$ of the desired equality into two pieces $\LS_1 - \LS_2$, where
\begin{align*}
\LS_1 &= (d+m)(u-m+2r)\sum_{k=0}^{d}g_k(u,d)p_k(m), \text{ and} \\
\LS_2 &= (d+m)(u-m+2r)\sum_{k=0}^{d}g_k(u,d)k p_{k-1}(m).
\end{align*}
Applying (\ref{eq:h-}) to $\LS_1$ with $h = u+2r$ yields
\[
\LS_1 = (d+m)\left[\sum_{k=0}^{d}g_k(u,d)(u+k+1+r)p_k(m)- \sum_{k=0}^d g_k(u,d)p_{k+1}(m)\right].
\] 
Then by applying (\ref{eq:h+}), this time with $h = d$, and collecting $p_\bullet(m)$ terms yields the following equivalent expression for $\LS_1$:
\begin{align}
\label{eq:LS1}
\sum_{k=0}^{d}g_k(u,d)\big[(d+r-k-1)(u+k+r)p_k(m) + (u-d+2k+2)p_{k+1}(m) - p_{k+2}(m)\big].
\end{align}
Similarly, $\LS_2$ simplifies to
\[
\sum_{k=0}^{d}g_k(u,d)k\big[(d-k+r)(u+k+r)p_{k-1}(m) + (u-d+2k+1)p_{k}(m) - p_{k+1}(m)\big].
\]
In order to realign the indices of the $p_{k-1}(m)$ terms, we use the identity
\[
\frac{g_{k-1}(u,d)}{g_k(u,d)} = \frac{k(d-k+r)(u+k-r)}{(d-k+1)(k-1+r)},
\]
which follows immediately from (\ref{eq:f}). Now, $\LS_2$ equals
\[
\sum_{k=0}^{d}g_k(u,d)\Big[\big[(d-k)(k+r)+ (k+1)(u-d+2k-1)\big]p_{k}(m) - (k+1)p_{k+1}(m)\Big].
\]
Subtracting this $\LS_2$ expression from the $\LS_1$ expression (\ref{eq:LS1}) completes the proof.
\end{proof}

The polynomials $\gamma^{1,i}_k, \gamma^{2,i}_k$ and $\gamma^{3,i}_k$ from Propositions \ref{prop:T1}, \ref{prop:T2} and \ref{prop:T3} (for $i \in \{0,1,2\}$) clearly satisfy $\gamma^{1,i}_k = \gamma^{2,i}_k + \gamma^{3,i}_k$ for all $k$. Relabeling the variable $d$ to $d-m$, it follows that the polynomials $\N^u_{d-m}$ satisfy the recurrence from Proposition \ref{prop:leaf3}; this observation concludes our proof of Theorem \ref{thm:p}.

\subsection{Symmetry between $u$ and $d$}

The polynomials $\P^u_d$ of Theorem \ref{thm:p} given by
\[
\P^u_d(m) = \sum^{d-m}_{k=0} \binom{d-m}{k}\prod_{i=0}^{k-1} (m+i+1-r)(i+r)\prod_{j=k+1}^{d-m}(d-m+r-j)(u+r+j)
\]
also enjoy a symmetry analogous to the one described in Theorem \ref{thm:symQ} for the polynomials $\Q^u_d$. In order to describe it, we define the family of functions
\[
\psi_k(x)= \prod_{i=1}^k (x+i)(i-1+2r).
\]

\begin{thm} 
The polynomials $\P^u_d$ and $\psi_k$ satisfy $\P^u_{d}(m)\psi_{u-m}(m)=\P^d_{u}(m)\psi_{d-m}(m)$.
\label{thm:symP}
\end{thm}
\begin{proof}
Using an argument similar to the one employed in Proposition \ref{prop:leaf1}, one can show that $\P^u_d$ satisfies the recurrence relation
\[
\beta'_0 \cdot \P^u_d(m) = \P^{u-1}_{d}(m-1) + \beta'_2\cdot\P^{u}_{d-1}(m),
\]
where 
\begin{align*}
\beta'_0(m,u,d,r) &= (u+d)(d-m+r), \text{ and}\\
\beta'_2(m,u,d,r) &= d(d-m)(u+d-m+r)(d-m-1+2r),
\end{align*}
along with the initial condition $\P^u_d(0)=\Q^{u}_d(0)$. To derive this recurrence, we construct the intermediate partitions $\kappa$ and $\nu$ by adding the leftmost box of $S_1$ and the rightmost box of $S_2$ to $\mu$ respectively before invoking the dual recurrence of Remark \ref{rem:leafrec}. We proceed by induction on $n=u+d$, and note that the result trivially holds when $n=0$. For $n>0$, we exchange the roles of $u$ and $d$ in the recurrence above to obtain
\[
\beta'_0(m,d,u,r)\cdot\P^d_u(m) = \P^{d-1}_{u}(m-1) + \beta'_2(m,d,u,r)\cdot\P^{d}_{u-1}(m).
\]
Multiplying throughout by $\psi_{d-m}(m)$ and using the identity
\[
\psi_{d-m}(m) = \frac{d}{m}\psi_{d-m}(m-1)
\]
with the first term on the right side, we note that $\beta'_0(m,d,u,r)\cdot\P^d_u(m)\psi_{d-m}(m)$ is given by
\[
\frac{d}{m}\cdot\P^{d-1}_{u}(m-1)\psi_{d-m}(m-1) + \beta'_2(m,d,u,r)\P^{d}_{u-1}(m)\psi_{d-m}(m). 
\]
By the inductive hypothesis, this expression equals
\[
\frac{d}{m} \P^{u}_{d-1}(m-1)\psi_{u-m+1}(m-1)+ \beta'_2(m,d,u,r)\P^{u-1}_{d}(m)\psi_{u-m-1}(m).
\]
Next, we shift the indices of the $\psi$-factors in each term above by using the identities
\begin{align*}
\psi_{u-m+1}(m-1) &= m(u-m+2r) \cdot \psi_{u-m}(m), \text{ and } \\ 
\psi_{u-m}(m) &= u(u-m-1+2r)\cdot \psi_{u-m-1}(m),
\end{align*}
and recall the definitions of the $\beta'_\bullet$ to obtain that $\beta'_0(m,d,u,r)\cdot \P^d_u(m)\psi_{d-m}(m)$ equals:
\[
\left[d(u-m+2r)\cdot\P^u_{d-1}(m-1)+ (u-m)(u+d-m+r)\cdot\P^{u-1}_d(m)\right]\psi_{d-m}(m).
\]
An appeal to Proposition \ref{prop:leaf3} confirms that the factor within the square brackets above equals $\beta_0(m,u,d,r)\cdot\P^u_d(m)$. Finally, we note that
\[
\beta_0(m,u,d,r) = (u+d)(u-m+r) =\beta'_0(m,d,u,r),
\] and this concludes the proof.
\end{proof}

\begin{cor} ${\P}^m_d(m)=\psi_{d-m}(m)$
\end{cor}

\begin{proof} If $u=m$, then Theorem \ref{thm:symP} gives us:
$\P^m_d(m)\psi_{m-m}(m)=\P^d_m(m)\psi_{d-m}(m).$
Since $\psi_{0}(m)=\P^d_m(m)=1$, we have ${\P}^m_d(m)=\psi_{d-m}(m)$.
\end{proof}

%
%
%
%
%
%

\end{document}